\documentclass[11pt,reqno]{amsart}
\usepackage{mathrsfs}
\usepackage[breaklinks]{hyperref}
\usepackage[headheight=110pt,top=1in, bottom=.9in, left=0.8in, right=0.8in]{geometry}

\usepackage{url}
\usepackage{amssymb}

\newcommand{\df}{\dfrac}
\newcommand{\tf}{\tfrac}

  \renewcommand{\a}{\alpha}
\renewcommand{\b}{\beta}

\newcommand{\g}{\gamma}
\newcommand{\G}{\Gamma}

\renewcommand{\(}{\left\(}
\renewcommand{\)}{\right\)}
\renewcommand{\[}{\left\[}
\renewcommand{\]}{\right\]}
\renewcommand{\i}{\infty}
\numberwithin{equation}{section}
 \theoremstyle{plain}
\newtheorem{theorem}{Theorem}[section]

\newtheorem{corollary}[theorem]{Corollary}

\newtheorem{remark}[]{Remark}

   \makeatletter
\def\proof{\@ifnextchar[{\@oproof}{\@nproof}}
\def\@oproof[#1][#2]{\trivlist\item[\hskip\labelsep\textit{#2 Proof of\
#1.}~]\ignorespaces}
\def\@nproof{\trivlist\item[\hskip\labelsep\textit{Proof.}~]\ignorespaces}

\makeatother

\makeatletter
\def\@tocline#1#2#3#4#5#6#7{\relax
  \ifnum #1>\c@tocdepth 
  \else
    \par \addpenalty\@secpenalty\addvspace{#2}%
    \begingroup \hyphenpenalty\@M
    \@ifempty{#4}{%
      \@tempdima\csname r@tocindent\number#1\endcsname\relax
    }{%
      \@tempdima#4\relax
    }%
    \parindent\z@ \leftskip#3\relax \advance\leftskip\@tempdima\relax
    \rightskip\@pnumwidth plus4em \parfillskip-\@pnumwidth
    #5\leavevmode\hskip-\@tempdima
      \ifcase #1
       \or\or \hskip 1em \or \hskip 2em \else \hskip 3em \fi%
      #6\nobreak\relax
    \dotfill\hbox to\@pnumwidth{\@tocpagenum{#7}}\par
    \nobreak
    \endgroup
  \fi}
\makeatother

\allowdisplaybreaks

\begin{document}
\title[modular transformation involving a generalized digamma function]{A modular relation involving a generalized digamma function and asymptotics of some integrals containing $\Xi(t)$}
\thanks{2020 \textit{Mathematics Subject Classification.} Primary 11M06, 41A60; Secondary 33E20.\\
\textit{Keywords and phrases.} Ramanujan's Lost notebook, modular relation, Generalized Hurwitz zeta function, Generalized digamma function, asymptotic estimates.}


\author{Atul Dixit}
\address{Discipline of Mathematics, Indian Institute of Technology Gandhinagar, Palaj, Gandhinagar 382355, Gujarat, India} 
\email{adixit@iitgn.ac.in}

\author{Rahul Kumar}
\address{Department of Mathematics, The Pennsylvania State University, University Park, PA 16802, U.S.A.} 
\email{rahuliitgn20@gmail.com}

\begin{abstract}
A modular relation of the form $F(\alpha, w)=F(\beta, iw)$, where $i=\sqrt{-1}$ and $\alpha\beta=1$, is obtained. It involves the generalized digamma function $\psi_w(a)$ which was recently studied by the authors in their work on developing the theory of the generalized Hurwitz zeta function $\zeta_w(s, a)$. The limiting case $w\to0$ of this modular relation is a famous result of Ramanujan on page $220$ of the Lost Notebook. We also obtain asymptotic estimate of a general integral involving the Riemann function  $\Xi(t)$ as $\alpha\to\infty$.  Not only does it give the asymptotic estimate of the integral occurring in our modular relation as a corollary but also some known results.
\end{abstract}
\maketitle
\vspace{-0.8cm}

\section{Introduction}
The Riemann functions $\xi(s)$ and $\Xi(t)$ are defined by \cite[p.~16]{titch}
\begin{align}\label{cdefn1}
\xi(s)=\frac{1}{2}s(s-1)\pi^{-\frac{s}{2}}\Gamma\left(\frac{s}{2}\right)\zeta(s),
\end{align}
and 
\begin{align}\label{cdefn2}
\Xi(t)=\xi\left(\frac{1}{2}+it\right),
\end{align}
where $\Gamma(s)$ and $\zeta(s)$ are Gamma and the Riemann zeta functions respectively.

Over the years, integrals comprising the Riemann $\Xi$-function in their integrands and their corresponding transformation formulas have found to be very useful in the theory of the Riemann zeta function $\zeta(s)$. Hardy \cite{Har} was one of the first mathematicians who realized the usefulness of such integrals when he proved his remarkable result that the infinitely many non-trivial zeros of $\zeta(s)$ lie on the critical line Re$(s)=1/2$. The crux of his argument relies on the identity
\begin{align}\label{thetaab}
\sqrt{\alpha}\bigg(\frac{1}{2\alpha}-\sum_{n=1}^{\infty}e^{-\pi\alpha^2n^2}\bigg)&=\sqrt{\beta}\bigg(\frac{1}{2\beta}-\sum_{n=1}^{\infty}e^{-\pi\beta^2n^2}\bigg)\nonumber\\
&=\frac{2}{\pi}\int_{0}^{\infty}\frac{\Xi(t/2)}{1+t^2}\cos\bigg(\frac{1}{2}t\log \a\bigg)\, dt,
\end{align} 
where $\a\b=1$ with Re$(\alpha^2)>0$ and Re$(\beta^2)>0$. Note that the first equality in the above identity is equivalent to the transformation formula for the Jacobi theta function:
\begin{equation*}
\sum_{n=-\infty}^{\infty}e^{-\pi\a^2n^2}=\frac{1}{\a}\sum_{n=-\infty}^{\infty}e^{-\pi n^2/\a^2} \quad (\mathrm{Re}(\alpha^2)>0).
\end{equation*}
A generalization of \eqref{thetaab} is given by \cite[Theorem 1.2]{dixit}
\begin{align}\label{genthetatr}
\sqrt{\alpha}\bigg(\frac{e^{-\frac{w^2}{8}}}{2\alpha}-e^{\frac{w^2}{8}}\sum_{n=1}^{\infty}e^{-\pi\alpha^2n^2}\cos(\sqrt{\pi}\alpha nw)\bigg)&=\sqrt{\beta}\bigg(\frac{e^{\frac{w^2}{8}}}{2\beta}-e^{-\frac{w^2}{8}}\sum_{n=1}^{\infty}e^{-\pi\beta^2n^2}\cosh(\sqrt{\pi}\beta nw)\bigg)\nonumber\\
&=\frac{1}{\pi}\int_{0}^{\infty}\frac{\Xi(t/2)}{1+t^2}\nabla\left(\alpha,w,\frac{1+it}{2}\right)\, dt,
\end{align}
where
\begin{align*}
\nabla(x,w,s)&:=\rho(x,w,s)+\rho(x,w,1-s),\nonumber\\
\rho(x,w,s)&:=x^{\frac{1}{2}-s}e^{-\frac{w^2}{8}}{}_1F_{1}\left(\frac{1-s}{2};\frac{1}{2};\frac{w^2}{4}\right),
\end{align*}
where ${}_1F_{1}(a;c;z):=\sum_{n=0}^\infty\frac{(a)_n}{(c)_n}\frac{z^n}{n!}$, $(a)_n=\Gamma(a+n)/\Gamma(a)$,  is the confluent hypergeometric function.

The first equality in \eqref{genthetatr} is due to Jacobi; the integral in \eqref{genthetatr} was found by the first author in \cite[Theorem 1.2]{dixit}. Applications of the above identity in generalizing Hardy's result on the infinitude of the zeros of $\zeta(s)$ can be found in \cite{dkmz} and \cite{drz}.

A year after Hardy's paper \cite{Har} appeared, Ramanujan also wrote a paper on some integrals containing the function $\Xi(t)$ in their integrands. One of the results that he provided in \cite{riemann} is 
\begin{align}\label{ramanujan z=0}
&\int_{0}^{\infty}\left|\Gamma\left(\frac{-1+it}{4}\right)
\Xi\left(\frac{t}{2}\right)\right|^2\frac{\cos (nt)}{1+t^2}\, dt=\pi^{3/2}\int_{0}^{\infty}\left(\frac{1}{\exp{(xe^n)}-1}-\frac{1}{xe^n}\right)\left(\frac{1}{\exp{(xe^{-n})}-1}-\frac{1}{xe^{-n}}\right)\, dx,
\end{align}
where $n\in\mathbb{R}^+$.

About the integral on the left-hand side of \eqref{ramanujan z=0}, Hardy \cite{ghh} says, \textit{``the properties of this integral resemble those of one which Mr. Littlewood and I have used, in a paper to be published shortly in Acta Mathematica to prove that\footnote{Note that there is a typo in this formula in that $\pi$ should not be present.}}
\begin{equation*}
\int_{-T}^{T}\left|\zeta\left(\frac{1}{2}+ti\right)\right|^2\, dt \sim
\frac{2}{\pi} T\log T\hspace{3mm}(T\to\infty)\textup{''}.
\end{equation*}
Very recently, Darses and Hillion \cite{dh} studied the polynomial moments with a weighted zeta square measure, namely,
\begin{align*}
\int_{-\infty}^\infty t^{2N}\left|\Gamma\left(\frac{1}{2}+it\right)\zeta\left(\frac{1}{2}+it\right)\right|^2dt.
\end{align*}
The starting point of their study is nothing but the Ramanujan's identity \eqref{ramanujan z=0}! Thus Hardy was indeed correct in his assessment of \eqref{ramanujan z=0}. 

In his Lost Notebook \cite{lnb}, Ramanujan provided a beautiful modular relation containing the integral on the left-hand side of \eqref{ramanujan z=0}:
\begin{align}\label{w1.26}
\sqrt{\a}\left\{\df{\gamma-\log(2\pi\a)}{2\a}+\sum_{n=1}^{\i}\phi(n\a)\right\}
&=\sqrt{\b}\left\{\df{\gamma-\log(2\pi\b)}{2\b}+\sum_{n=1}^{\i}\phi(n\b)\right\}\nonumber\\
&=-\df{1}{\pi^{3/2}}\int_0^{\i}\left|\Xi\left(\df{1}{2}t\right)\Gamma\left(\df{-1+it}{4}\right)\right|^2
\df{\cos\left(\tf{1}{2}t\log\a\right)}{1+t^2}\, dt,
\end{align}
where $\alpha,\ \beta>0$ such that $\alpha\beta=1$, and $\phi(x):=\psi(x)+\frac{1}{2x}-\log(x)$, with $\psi(x)=\Gamma'(x)/\Gamma(x)$ being the digamma function. By a modular relation, we mean a transformation of the form $F(-1/z)=F(z)$, 
where $z$  is in the upper-half plane. Unlike modular forms,  
these transformations may not be governed by the translation $z\to z+1$. 
Note that $z\to-1/z$ can be equivalently rewritten in the form $\alpha\to\beta$ with 
$\alpha\beta=1$, where Re$(\alpha)>0$ and Re$(\beta)>0$. For more details, see  \cite[p.~49]{dixitms}.

The first ever proof of this identity appeared in \cite{bcbad}. Guinand \cite{guinand} rediscovered the first equality and remarked that ``\emph{This formula also seems to have been overlooked}". Very recently, Gupta and the second author \cite{gk} showed that this formula of Ramanujan fits in the theory of the Herglotz function.

One of the goals of this paper is to provide a new generalization of \eqref{w1.26}. Before we do this though, we first record an existing generalization of \eqref{w1.26} given by the first author \cite[Theorem 1.4]{dixitijnt}.
\begin{theorem}\label{sechur}
Let $-1<$ \textup{Re} $z<1$. Define $\varphi(z,x)$ by
\begin{equation}\label{varphi1}
\varphi(z,x):=\zeta(z+1,x)-\frac{1}{z}x^{-z}-\frac{1}{2}x^{-z-1},
\end{equation}
where $\zeta(z,x)$ denotes the Hurwitz zeta function. Then if $\alpha$ and $\beta$ are any positive numbers such that $\alpha\beta=1$,
\begin{align}\label{mainneq2}
&\a^{\frac{z+1}{2}}\left(\sum_{n=1}^{\infty}\varphi(z,n\a)-\frac{\zeta(z+1)}{2\alpha^{z+1}}-\frac{\zeta(z)}{\alpha z}\right)=\b^{\frac{z+1}{2}}\left(\sum_{n=1}^{\infty}\varphi(z,n\b)-\frac{\zeta(z+1)}{2\beta^{z+1}}-\frac{\zeta(z)}{\beta z}\right)\nonumber\\
&=\frac{8(4\pi)^{\frac{z-3}{2}}}{\Gamma(z+1)}\int_{0}^{\infty}\Gamma\left(\frac{z-1+it}{4}\right)\Gamma\left(\frac{z-1-it}{4}\right)
\Xi\left(\frac{t+iz}{2}\right)\Xi\left(\frac{t-iz}{2}\right)\frac{\cos\left( \tf{1}{2}t\log\a\right)}{(z+1)^2+t^2}\, dt.
\end{align}
\end{theorem}
Note that if we let $z\to0$ in the theorem above, we get \eqref{w1.26}. For more details, see \cite[pp.~1163-1165]{dixitijnt}.

Recently, in their quest for putting the theta structure on the modular relation \eqref{mainneq2}, the current authors \cite{dkRiMS1} introduced a new generalization of the Hurwitz zeta function, which is stated next. Let $\mathfrak{B}:=\{\xi:\textup{Re}(\xi)=1, \textup{Im}(\xi)\neq 0\}$. The generalized Hurwitz zeta function for $w\in\mathbb{C}\backslash\{0\}$, Re$(s)>1$ and $a\in\mathbb{C}\backslash\mathfrak{B}$ is defined by \cite[Equation (1.1.14)]{dkRiMS1}
\begin{align}\label{new zeta}
\zeta_w(s, a)&:=\frac{4}{w^2\sqrt{\pi}\G\left(\frac{s+1}{2}\right)}\sum_{n=1}^{\infty}\int_{0}^{\infty}\int_{0}^{\infty}\frac{(uv)^{s-1}e^{-(u^2+v^2)}\sin(wv)\sinh(wu)}{\left(n^2u^2+(a-1)^2v^2\right)^{s/2}}\, dudv.
\end{align}
The function $\zeta_w(s, a)$ satisfies several interesting properties, one of them being $\zeta_w(s,a)=\zeta_w(s,2-a)$ for Re$(s)>1$. Observe that the symmetry in the variable $a$ here along the line Re$(a)=1$ does not hold for the Hurwitz zeta function $\zeta(s, a)$. The reader is referred to \cite{dkRiMS1} for more properties of $\zeta_w(s, a)$. 

Note that for \textup{Re}$(s)>1$ \cite[Theorem 1.1.2]{dkRiMS1}, $\zeta_w(s, a)$ reduces to the Hurwitz zeta function $\zeta(s, a)$ as follows:
\begin{equation*}
\lim_{w\to0}\zeta_w(s, a)=\begin{cases}
\zeta(s,a) & \text{if\ $\mathrm{Re}(a)>1$}, \\
   \zeta(s, 2-a) & \text{if\ $\mathrm{Re}(a)<1$}. 
\end{cases}
\end{equation*}
In fact, the Hurwitz zeta function is the constant term in the Taylor series expansion of $\zeta_w(s, a)$ around $w=0$.
 
In \cite[Theorem 1.1.3]{dkRiMS1}, it was shown that $\zeta_w(s,a)$ has meromorphic continuation in the region Re$(s)>-1$ with a simple pole at $s=1$. To record this result in the following theorem, we need the error function erf$(w)$  and the imaginary error function erfi$(w)$:
\begin{align}
\mathrm{erf}(w)&:=\frac{2}{\sqrt{\pi}}\int_0^w e^{-t^2}dt=\frac{2w}{\sqrt{\pi}}e^{-w^2}{}_1F_1\left(1;\frac{3}{2};w^2\right),\label{erf funct}\\
\mathrm{erfi}(w)&:=\frac{2}{\sqrt{\pi}}\int_0^w e^{t^2}dt=\frac{2w}{\sqrt{\pi}}e^{w^2}{}_1F_1\left(1;\frac{3}{2};-w^2\right).\label{erfi funct}
\end{align} 
 
\begin{theorem}\label{lse}
Let\footnote{There is a typo in the statement of Theorem 1.1.3 of \cite{dkRiMS1}. The condition should be $a>0$ and not $0<a<1$.} $a>0$ and $w\in\mathbb{C}$. The generalized Hurwitz zeta function $\zeta_w(s, a)$ can be analytically continued to $\textup{Re}(s)>-1$ except for a simple pole at $s=1$. The residue at this pole is
\begin{equation}\label{equv}
e^{\frac{w^2}{4}}{}_1F_{1}^{2}\left(1;\frac{3}{2};-\frac{w^2}{4}\right)=\frac{\pi}{w^2}e^{-\frac{w^2}{4}}\textup{erfi}^{2}\left(\frac{w}{2}\right),
\end{equation}
where $\textup{erfi}(w)$ is defined in \eqref{erfi funct}. Moreover near $s=1$, we have
\begin{align}\label{luslus}
\zeta_w(s, a+1)=\frac{e^{\frac{w^2}{4}}{}_1F_{1}^{2}\left(1;\frac{3}{2};-\frac{w^2}{4}\right)}{s-1}-\psi_w(a+1) +O_{w, a}(|s-1|),
\end{align}
where $\psi_w(a)$ is a new generalization of the digamma function $\psi(a)$ defined by
\begin{align*}
\psi_w(a):=\frac{4}{w^2\sqrt{\pi}}\int_0^\infty\int_0^\infty\int_0^\infty \frac{e^{-(u^2+v^2+x)}}{u}\sin(wv)\sinh(wu)\left(\frac{1}{x}-\frac{J_{0}\left((a-1)\frac{vx}{u}\right)}{1-e^{-x}}\right)dxdudv,
\end{align*}	
with $J_z(x)$ being the Bessel function of the first kind \cite[p.~40]{watson-1966a}. 
\end{theorem}
Although the definition of $\psi_w(a)$ looks complicated, it is pleasing to see that Ramanujan's formula \eqref{ramanujan z=0} can be extended in the setting of $\psi_w(a)$. This is done next and is the main result of this paper.
\begin{theorem}\label{ramanujanw}
Let $w\in\mathbb{C}$ and $x>0$. Let $\mathrm{erf}(w)$ be defined in \eqref{erf funct}. Define 
\begin{align}\label{lamdaw}
\lambda_w(x):=\psi_w(x+1)-\frac{1}{2x}\mathcal{C}(w)-\mathcal{C}(iw)\log(x)-\frac{1}{2}\mathcal{B}(iw),
\end{align}
where 
\begin{align}\label{cw}
\mathcal{C}(w):=\frac{\pi}{w^2} e^{\frac{w^2}{4}}\mathrm{erf}^2\left(\frac{w}{2}\right),
\end{align}
and 
\begin{align}\label{bw}
\mathcal{B}(w):=\frac{\sqrt{\pi}}{w}\mathrm{erf}\left(\frac{w}{2}\right)\sum_{n=0}^\infty\frac{(w^2/4)^n}{\left(3/2\right)_n}(\psi(n+1)+\gamma).
\end{align}
Then for any positive integers $\a, \b$ such that $\a\b=1$,
\begin{align}\label{ramanujanweqn}
&\sqrt{\alpha}\left\{\frac{\gamma-\log(2\pi\alpha)}{2\alpha}\mathcal{C}(w)+\frac{1}{2\alpha} \mathcal{B}(w)+\sum_{m=1}^\infty \lambda_w(m\alpha)\right\}\nonumber \\
&=\sqrt{\beta}\left\{\frac{\gamma-\log(2\pi\beta)}{2\beta}\mathcal{C}(iw)+\frac{1}{2\beta} \mathcal{B}(iw)+\sum_{m=1}^\infty \lambda_{iw} (m\beta)\right\}\nonumber\\
&=-\frac{e^{\frac{w^2}{4}}}{2\pi^{\frac{3}{2}}}\int_0^\infty\left|\Gamma\left(\frac{-1+it}{4}\right)\right|^2\frac{\Xi^{2}\left(\frac{t}{2}\right)}{1+t^2}\left(\a^{\frac{it}{2}}{}_1F_{1}^{2}\left(\frac{3+it}{4};\frac{3}{2};-\frac{w^2}{4}\right)+\a^{-\frac{it}{2}}{}_1F_{1}^{2}\left(\frac{3-it}{4};\frac{3}{2};-\frac{w^2}{4}\right)\right)\, dt.
\end{align}
\end{theorem}
Note that the modular relation in the first equality given above is of the form $F(\alpha, w)=F(\beta, iw)$, and hence is analogous to the general theta transformation formula \eqref{genthetatr}. See \cite{dixit} for more details.
 
Theorem \ref{ramanujanw} gives Ramanujan's formula \eqref{w1.26} as a special case:
\begin{corollary}\label{w=0 of main thm}
Equation \eqref{w1.26} holds true.
\end{corollary}

Asymptotic estimates for integrals containing Riemann $\Xi$-function are useful in the study of the moments of the Riemann zeta function. For example, Hardy and Littlewood \cite[p.~151, Section 2.4]{hl} discovered the asymptotic formula for the second moment of the Riemann zeta function through the asymptotic estimate of the integral 
$$\int_0^\infty\left\{\frac{\Xi(t)}{1+t^2}\right\}^2e^{2\alpha t}dt, \qquad\mathrm{as}\ \alpha\to\infty.$$
See \cite{dh} for a more recent application of such integrals. An asymptotic result for a smoothly weighted second moment of the Riemann zeta function, also containing the cosine term as in \eqref{w1.26} and generalizing a result of Hardy, was obtained by Maier \cite[Equation (9)]{maieruber}, namely, he showed that for coprime integers $p$ and $q$, 
\begin{equation*}
	\lim_{\delta\to0}\frac{\delta}{\log\left(\frac{1}{\delta}\right)}\int_{0}^{\infty}e^{-\delta t}\left|\zeta\left(\frac{1}{2}+it\right)\right|^{2}\cos\left(t\log\left(p/q\right)\right)\, dt=(pq)^{-1/2}.
\end{equation*}

We next derive an asymptotic estimate of a general integral containing Riemann $\Xi$-function.
\begin{theorem}\label{asymp of int}
Let $A_w(z)$ be defined by
\begin{align}\label{ab}
A_w(z)&:=\frac{\sqrt{\pi}}{w}\textup{erf}\left(\frac{w}{2}\right){}_1F_{1}\left(1+\frac{z}{2};\frac{3}{2};\frac{w^2}{4}\right)=e^{-\frac{w^2}{4}}{}_1F_{1}\left(1;\frac{3}{2};\frac{w^2}{4}\right){}_1F_{1}\left(1+\frac{z}{2};\frac{3}{2};\frac{w^2}{4}\right).
\end{align}
Let $m\in\mathbb{N},\ w\in\mathbb{C}$ and $-1<\mathrm{Re}(z)<1$. Then as $\alpha\to\infty$, we have
\begin{align}\label{asymp of int eqn}
&\int_0^\infty \Gamma\left(\frac{z-1+it}{4}\right)\Gamma\left(\frac{z-1-it}{4}\right)\Xi\left(\frac{t+iz}{2}\right)\Xi\left(\frac{t-iz}{2}\right)\frac{\Delta_2\left(\alpha,\frac{z}{2},w,\frac{1+it}{2}\right)}{(z+1)^2+t^2}dt\nonumber\\
&=-\frac{\G(z+1)}{2^{z-1}\pi^{\frac{z-3}{2}}}\left(\frac{\zeta(z+1)A_w(z)}{2\alpha^{\frac{z+1}{2}}}+\frac{\zeta(z)A_w(-z)}{ z\alpha^{\frac{1-z}{2}}}\right)\nonumber\\
&\quad-\frac{e^{-\frac{w^2}{4}}\alpha^{\frac{1-z}{2}}}{2^{z-2}\pi^{\frac{z-3}{2}}}\sum_{k=1}^{m-1}\frac{(-1)^k\Gamma\left(z+2k\right)}{(2\pi\alpha)^{2k}}\zeta(2k)\zeta(z+2k){}_1F_1\left(\frac{z+1}{2}+k;\frac{3}{2};\frac{w^2}{4}\right){}_1F_1\left(\frac{1}{2}+k;\frac{3}{2};\frac{w^2}{4}\right)\nonumber\\
&\quad+O_{z,m,w}\left(\alpha^{-\frac{1}{2}\mathrm{Re}(z)-2m}\right), 
\end{align}
where,
\begin{align}\label{Del}
\Delta_{2}(x, z, w, s)&:=\omega(x, z, w, s)+\omega(x, z, w, 1-s),\nonumber\\
\omega(x, z, w, s)&:=e^{\frac{w^2}{4}}x^{\frac{1}{2}-s}{}_1F_{1}\left(1-\frac{s+z}{2};\frac{3}{2};-\frac{w^2}{4}\right){}_1F_{1}\left(1-\frac{s-z}{2};\frac{3}{2};-\frac{w^2}{4}\right).
\end{align}
\end{theorem}
The special case $w=0$ of the above theorem is a known result \cite[Theorem 1.10]{drzacta}:
\begin{corollary}\label{drz asym}
Let $-1<\mathrm{Re}(z)<1$ and $m\in\mathbb{N}$. As $\alpha\to\infty$,
\begin{align}\label{drz asym eqn}
&\int_0^\infty \Gamma\left(\frac{z-1+it}{4}\right)\Gamma\left(\frac{z-1-it}{4}\right)\Xi\left(\frac{t+iz}{2}\right)\Xi\left(\frac{t-iz}{2}\right)\frac{\cos\left(\frac{t}{2}\log(\alpha)\right)}{(z+1)^2+t^2}dt\nonumber\\
&=-\frac{\G(z+1)}{2^{z-1}\pi^{\frac{z-3}{2}}}\left(\frac{\zeta(z+1)}{2\alpha^{\frac{z+1}{2}}}+\frac{\zeta(z)}{ z\alpha^{\frac{1-z}{2}}}\right)-\frac{\alpha^{\frac{1-z}{2}}}{2^{z-2}\pi^{\frac{z-3}{2}}}\sum_{k=1}^{m-1}\frac{(-1)^k\Gamma\left(z+2k\right)}{(2\pi\alpha)^{2k}}\zeta(2k)\zeta(z+2k)\nonumber\\
&\quad+O_{z,m}\left(\alpha^{-\frac{1}{2}\mathrm{Re}(z)-2m}\right).
\end{align}
\end{corollary}
Moreover, letting  $z=0$ in Theorem \ref{asymp of int} gives the asymptotic estimate of the integral appearing in Theorem \ref{ramanujanw}.
\begin{corollary}\label{z=0 asym}
Let $\mathcal{C}(w)$ and $\mathcal{B}(w)$ be defined in \eqref{cw} and \eqref{bw} respectively. Let $w\in\mathbb{C}$ and $m\in\mathbb{N}$. As $\alpha\to\infty$, we have
\begin{align}\label{z=0 asym eqn}
&\int_0^\infty\left|\Gamma\left(\frac{-1+it}{4}\right)\right|^2\frac{\Xi^{2}\left(\frac{t}{2}\right)}{1+t^2}\left(\a^{\frac{it}{2}}{}_1F_{1}^{2}\left(\frac{3+it}{4};\frac{3}{2};-\frac{w^2}{4}\right)+\a^{-\frac{it}{2}}{}_1F_{1}^{2}\left(\frac{3-it}{4};\frac{3}{2};-\frac{w^2}{4}\right)\right)\, dt\nonumber\\
&=-2e^{-\frac{w^2}{4}}\pi^{\frac{3}{2}}\sqrt{\alpha}\left(\frac{\gamma-\log(2\pi\alpha)}{2\alpha}\mathcal{C}(w)+\frac{1}{2\alpha}\mathcal{B}(w)\right)\nonumber\\
&\qquad-4e^{-\frac{w^2}{2}}\pi^{\frac{3}{2}}\sqrt{\alpha}\sum_{k=1}^{m-1}\frac{(-1)^k}{(2\pi\alpha)^{2k}}\Gamma(2k)\zeta^2(2k){}_1F_1^2\left(k+\frac{1}{2};\frac{3}{2};\frac{w^2}{4}\right)+O_{m,w}\left(\alpha^{-2m}\right).
\end{align}
\end{corollary}
Putting $w=0$ in Corollary \ref{z=0 asym} then gives us a known result due to Oloa \cite[Equation (1.5)]{oloa}.

\begin{remark}
At first glance, letting $z=0$ in Theorem \ref{asymp of int} seems problematic. However, the function $\frac{1}{2}\alpha^{\frac{-z-1}{2}}\zeta(z+1)A_w(z)+\frac{1}{ z}\alpha^{\frac{z-1}{2}}\zeta(z)A_w(-z)$ has a removable singularity at $z=0$. See \eqref{tbr}.
\end{remark}

It is important to note that the results on asymptotics stated above are usually obtained by employing Watson's lemma \cite[p.~32]{temme}; see for example \cite{drzacta}. However, in the general setting with the parameter $w$ which we are considering here, Watson's lemma is inapplicable. Nevertheless, we are able to obtain an asymptotic estimate for the integral given in \eqref{asymp of int eqn}, and that too by elementary means.  See Section \ref{asymptotics} for details.

This paper is organized as follows. Section \ref{thm1.3} is devoted to proving Theorem \ref{ramanujanw}. Theorem \ref{asymp of int} and its corollaries are proved in Section \ref{asymptotics}.

\section{Proof of Theorem \textup{\ref{ramanujanw}}}\label{thm1.3}

We obtain \eqref{ramanujanweqn} as a special case of a more general result derived in \cite[Theorems 1.1.5, 1.1.6 ]{dkRiMS1}, which is stated next. It will be clear from the proof that deriving \eqref{ramanujanweqn} from a generalization is not straightforward.
\begin{theorem}\label{genramhureq}
Let $A_w(z)$ and $\Delta_2$ be defined in \eqref{ab} and \eqref{Del} respectively. Let $\zeta_w(s,a)$ be the generalized Hurwitz zeta function defined in \eqref{new zeta}. Let $w\in\mathbb{C}$, $-1<$ \textup{Re}$(z)<1$ and $x>0$. Define $\varphi_w(z, x)$ by
\begin{align}\label{varphi}
\varphi_w(z, x)&:=\zeta_w(z+1, x+1)+\frac{1}{2}A_w(z)x^{-z-1}-A_{iw}(-z)\frac{x^{-z}}{z}.
\end{align}
Then for $\a, \b>0$ and $\a\b=1$, we have
\begin{align}\label{genramhureqeqn}
&\a^{\frac{z+1}{2}}\left(\sum_{m=1}^{\infty}\varphi_w(z,m\a)-\frac{\zeta(z+1)A_w(z)}{2\alpha^{z+1}}-\frac{\zeta(z)A_w(-z)}{\alpha z}\right)\nonumber\\
&=\b^{\frac{z+1}{2}}\left(\sum_{m=1}^{\infty}\varphi_{iw}(z,m\b)-\frac{\zeta(z+1)A_{iw}(z)}{2\beta^{z+1}}-\frac{\zeta(z)A_{iw}(-z)}{\beta z}\right)\nonumber\\
&=\frac{2^{z-1}\pi^{\frac{z-3}{2}}}{\Gamma(z+1)}\int_0^\infty \Gamma\left(\frac{z-1+it}{4}\right)\Gamma\left(\frac{z-1-it}{4}\right)\Xi\left(\frac{t+iz}{2}\right)\Xi\left(\frac{t-iz}{2}\right)\frac{\Delta_2\left(\alpha,\frac{z}{2},w,\frac{1+it}{2}\right)}{(z+1)^2+t^2}dt.
\end{align}
\end{theorem}

\begin{proof}[Theorem \textup{\ref{ramanujanw}}][]
It can be seen that the series containing $\lambda_w(m\alpha)$($\lambda_w(m\beta)$) is convergent using \cite[p.~42, Remark 9]{dkRiMS1}:
\begin{align*}
\lambda_w(m\alpha)=\psi_w(m\alpha+1)-\frac{1}{2m\alpha}\mathcal{C}(w)-\mathcal{C}(iw)\log(m\alpha)-\frac{1}{2}\mathcal{B}(iw)=O_w\left(\frac{1}{m^2}\right).
\end{align*}
We now prove \eqref{ramanujanweqn}. Let $z\to 0$ in \eqref{genramhureqeqn}. Then the right-hand side of \eqref{genramhureqeqn} is easily seen to reduce to the extreme right side of \eqref{ramanujanweqn} (up to a minus sign). Thus we need only show that 
\begin{align}\label{claim}
&\lim_{z\to 0}\left(\sum_{m=1}^{\infty}\varphi_w(z,m\a)-\frac{\zeta(z+1)A_w(z)}{2\alpha^{z+1}}-\frac{\zeta(z)A_w(-z)}{\alpha z}\right)\nonumber\\
&=-\left(\frac{\gamma-\log(2\pi\alpha)}{2\alpha}\mathcal{C}(w)+\frac{\sqrt{\pi}}{2\alpha} \mathcal{B}(w)+\sum_{m=1}^\infty \lambda_w(m\alpha)\right).
\end{align}
 We first evaluate
\begin{align}
L_1(w,\a):&=\lim_{z\rightarrow 0}\left\{\frac{\zeta(z+1)A_w(z)}{2\alpha^{z+1}}+\frac{\zeta(z)A_w(-z)}{\alpha z}\right\}\label{lim-1}\\
&=\frac{\sqrt{\pi}}{w}\mathrm{erf}\left(\frac{w}{2}\right)\lim_{z\rightarrow 0}\left\{\frac{\zeta(1+z)}{2\alpha^{1+z}}{}_1F_1\left(1+\frac{z}{2};\frac{3}{2};\frac{w^2}{4}\right)+\frac{\zeta(z)}{\alpha z}{}_1F_1\left(1-\frac{z}{2};\frac{3}{2};\frac{w^2}{4}\right)\right\},\label{lim1}
\end{align}
where in the last step we used \eqref{ab}. Note that \cite[p.~16, Equation (2.1.16)]{titch} near $s=1$,
\begin{equation}\label{zetalaurent}
\zeta(s)=\frac{1}{s-1}+\g+O(|s-1|),
\end{equation}
whence
\begin{align}
\lim_{z\rightarrow 0}\left(\zeta(z+1)-\frac{1}{z}\right)=\gamma.\nonumber
\end{align}
Also, the following series expansions hold as $z\to 0$:
\begin{align}\label{serexp}
\alpha^{-z}&=1-z\log\alpha+\frac{z^2(\log\alpha)^2}{2!}+O(|z|^3),\nonumber\\
\zeta(z)&=-\frac{1}{2}-\frac{1}{2}\log(2\pi)z+O(|z|^2),\nonumber\\
{}_1F_1\left(1\pm\frac{z}{2};\frac{3}{2};\frac{w^2}{4}\right)&=\frac{\sqrt{\pi} e^{\frac{w^2}{4}}\mathrm{erf}\left(\frac{w}{2}\right)}{w}+z\frac{d}{dz}{}_1F_1\left(1\pm\frac{z}{2};\frac{3}{2};\frac{w^2}{4}\right)\Bigg|_{z=0} +\frac{z^2}{2!}\frac{d^2}{dz^2}{}_1F_1\left(1\pm\frac{z}{2};\frac{3}{2};\frac{w^2}{4}\right)\Bigg|_{z=0}\nonumber\\
&\qquad+O(|z|^3).
\end{align}
Hence from \eqref{lim1}, \eqref{zetalaurent} and \eqref{serexp}, we compute
\begin{align}\label{L 1}
L_1(w,\a)&=\frac{\sqrt{\pi}}{w}\mathrm{erf}\left(\frac{w}{2}\right)\lim_{z\rightarrow 0}\Bigg\{\frac{\left(\frac{1}{z}+\gamma-\gamma_1z+O(|z|^2)\right)}{2\alpha}\left(1-z\log\alpha+\frac{z^2}{2!}(\log\alpha)^2+O_{\a}(|z|^3)\right)\nonumber\\
&\times\left(\frac{\sqrt{\pi} e^{\frac{w^2}{4}}\mathrm{erf}\left(\frac{w}{2}\right)}{w}+z\frac{d}{dz}{}_1F_1\left(1+\frac{z}{2};\frac{3}{2};\frac{w^2}{4}\right)\Bigg|_{z=0}+O_{w}(|z|^2)\right) \nonumber\\
&+\frac{\left(-\frac{1}{2}-\frac{1}{2}\log(2\pi)z+O(|z|^2)\right)}{\alpha z}\left(\frac{\sqrt{\pi}e^{\frac{w^2}{4}}\mathrm{erf}\left(\frac{w}{2}\right)}{w}+z\frac{d}{dz}{}_1F_1\left(1-\frac{z}{2};\frac{3}{2};\frac{w^2}{4}\right)\Bigg|_{z=0}+O_{w}(|z|)^2\right)\Bigg\}\nonumber\\
=&\frac{\sqrt{\pi}}{w}\mathrm{erf}\left(\frac{w}{2}\right)\left\{\frac{\sqrt{\pi} e^{\frac{w^2}{4}}\mathrm{erf}\left(\frac{w}{2}\right)}{w}\frac{\gamma}{2\alpha}-\frac{\log(2\pi\alpha)}{2\alpha}\frac{\sqrt{\pi} e^{\frac{w^2}{4}}\mathrm{erf}\left(\frac{w}{2}\right)}{w}+\frac{1}{2\alpha}\Bigg(\frac{d}{dz}{}_1F_1\left(1+\frac{z}{2};\frac{3}{2};\frac{w^2}{4}\right)\Bigg|_{z=0}\right.\nonumber\\
&\left.\quad-\frac{d}{dz}{}_1F_1\left(1-\frac{z}{2};\frac{3}{2};\frac{w^2}{4}\right)\Bigg|_{z=0}\Bigg)\right\}.
\end{align}
Note that
\begin{align*}
2\frac{d}{dz}{}_1F_1\left(1+\frac{z}{2};\frac{3}{2};\frac{w^2}{4}\right)&=2\frac{d}{dz}\sum_{n=0}^\infty \frac{\left(1+\frac{z}{2}\right)_n}{\left(3/2\right)_n}\frac{\left(w^2/4\right)^n}{n!} \nonumber\\
&=2\sum_{n=0}^\infty \frac{\left(w^2/4\right)^n}{\left(3/2\right)_nn!}\frac{d}{dz}\prod_{j=0}^{n-1}\left(1+\frac{z}{2}+j\right) \nonumber\\
&=\sum_{n=0}^\infty \frac{\left(w^2/4\right)^n}{\left(3/2\right)_nn!}\left(1+\frac{z}{2}\right)_n \sum_{j=0}^{n-1}\frac{1}{\left(1+\frac{z}{2}+j\right)},
\end{align*}
so that
\begin{align}\label{1F1 +}
2\frac{d}{dz}{}_1F_1\left(1+\frac{z}{2};\frac{3}{2};\frac{w^2}{4}\right)\Bigg|_{z=0}&=\sum_{n=0}^\infty \frac{\left(w^2/4\right)^n}{\left(3/2\right)_n}\left(\psi(n+1)+\gamma\right),
\end{align}
where we use the well-known result \cite[p.~54, Equation (3.10)]{temme} for $n\in\mathbb{N}\cup\{0\}$:
\begin{equation*}
\psi(n+1)=-\g+\sum_{m=0}^{\infty}\left(\frac{1}{m+1}-\frac{1}{n+1+m}\right).
\end{equation*}
Similarly,
\begin{align}\label{1F1 -}
-2\frac{d}{dz}{}_1F_1\left(1-\frac{z}{2};\frac{3}{2};\frac{w^2}{4}\right)\Bigg|_{z=0}=\sum_{n=0}^\infty \frac{\left(w^2/4\right)^n}{\left(3/2\right)_n}(\psi(n+1)+\gamma).
\end{align}
Therefore, from \eqref{L 1}, \eqref{1F1 +}, \eqref{1F1 -},  we have
\begin{align}\label{final L with alpha}
L_1(w,\a)
&=\frac{\gamma-\log(2\pi\alpha)}{2\alpha}\frac{\pi}{w^2} e^{\frac{w^2}{4}}\mathrm{erf}^2\left(\frac{w}{2}\right)
+\frac{1}{2\alpha}\frac{\sqrt{\pi}}{w}\mathrm{erf}\left(\frac{w}{2}\right)\sum_{n=0}^\infty \frac{(w^2/4)^n}{\left(3/2\right)_n}(\psi(n+1)+\gamma)\nonumber\\
&=\frac{\gamma-\log(2\pi\alpha)}{2\alpha}\mathcal{C}(w)
+\frac{1}{2\alpha}\mathcal{B}(w),
\end{align}
using the definitions in \eqref{cw} and \eqref{bw}. Now the asymptotic expansion of $\zeta_{w}(s, a)$ \cite[p.~38, Theorem 5.3.1]{dkRiMS1}, for $-1<\textup{Re}(s)<2$, 
\begin{align}\label{hermiteeqn1}
\zeta_w(s,a+1)=-\frac{a^{-s}}{2}A_w(s-1)+\frac{a^{1-s}}{s-1}A_{iw}(1-s)+O_{s,w}\left(a^{-\textup{Re}(s)-1}\right), \ \mathrm{as}\ a\to\infty,
\end{align}
implies that the series in \eqref{genramhureqeqn} are absolutely and uniformly convergent in a neighborhood of $z=0$. Hence, we can take the limit inside the series on the extreme left-hand side of \eqref{genramhureqeqn}. This gives, using \eqref{varphi}, 
\begin{align}\label{seplim}
\lim_{z\to 0}\sum_{m=1}^\infty\varphi_w(z,m\alpha)
&=\sum_{m=1}^\infty\lim_{z\to 0}\left\{\zeta_w(z+1,m\alpha+1)-\frac{\sqrt{\pi}\mathrm{erfi}\left(\frac{w}{2}\right){}_1F_1\left(1-\frac{z}{2};\frac{3}{2};-\frac{w^2}{4}\right)}{z(m\alpha)^zw}\right.\nonumber\\
&\left.\qquad+\frac{\sqrt{\pi}\mathrm{erf}\left(\frac{w}{2}\right){}_1F_1\left(1+\frac{z}{2};\frac{3}{2};\frac{w^2}{4}\right)}{2(m\alpha)^{z+1}w}\right\}\nonumber\\
&=\sum_{m=1}^\infty\left(\lim_{z\to 0}\left\{\left(\zeta_w(z+1,m\alpha+1)-\frac{\pi e^{-\frac{w^2}{4}}\mathrm{erfi}^2\left(\frac{w}{2}\right)}{zw^2}\right)\right.\right.\nonumber\\
&\left.\left.\qquad+\frac{\sqrt{\pi}\mathrm{erf}\left(\frac{w}{2}\right){}_1F_1\left(1+\frac{z}{2};\frac{3}{2};\frac{w^2}{4}\right)}{2(m\alpha)^{z+1}w}\right\}+L_{2}(w, \a, m)\right),
\end{align}
where $L_2(w, \a, m)$ is defined as
\begin{align}\label{l2wam}
L_{2}(w, \a, m):=\lim_{z\rightarrow 0}\left(\frac{\pi e^{-\frac{w^2}{4}}\mathrm{erfi}^2\left(\frac{w}{2}\right)}{zw^2}-\frac{\sqrt{\pi}\mathrm{erfi}\left(\frac{w}{2}\right)}{z(m\alpha)^zw}{}_1F_1\left(1-\frac{z}{2};\frac{3}{2};-\frac{w^2}{4}\right)\right). 
\end{align}
Invoking \eqref{equv}, \eqref{luslus} in \eqref{seplim}, we are led to
\begin{align}\label{lim in terms of l2}
\lim_{z\to 0}\sum_{m=1}^\infty\varphi_w(z,m\alpha)
&=\sum_{m=1}^\infty\left\{-\psi_w(m\alpha+1)+\frac{\mathcal{C}(w)}{2m\alpha}+L_{2}(w, \a, m)\right\}.
\end{align}
We now use \eqref{L 1} in \eqref{l2wam} to see that
\begin{align}\label{lim of l2}
L_{2}(w, \a, m)&=\lim_{z\rightarrow 0}\left\{\frac{\sqrt{\pi}\mathrm{erfi}(w/2)}{w}\left[\frac{e^{-w^2/4}\sqrt{\pi}\mathrm{erfi}(w/2)}{zw}-\left(1-z\log(m\alpha)+O(|z|^2)\right)\right.\right. \nonumber\\
&\left.\left.\qquad\times\left(\frac{\sqrt{\pi}e^{-\frac{w^2}{4}}\mathrm{erfi}\left(\frac{w}{2}\right)}{wz}+\frac{d}{dz}{}_1F_1\left(1-\frac{z}{2};\frac{3}{2};-\frac{w^2}{4}\right)\Bigg|_{z=0}+O(|z|)\right)\right]\right\}\nonumber\\
&=\frac{\sqrt{\pi}\mathrm{erfi}(w/2)}{w}\left(\frac{\sqrt{\pi}}{w}e^{-\frac{w^2}{4}}\mathrm{erfi}\left(\frac{w}{2}\right)\log(m\alpha)-\frac{d}{dz}{}_1F_1\left(1-\frac{z}{2};\frac{3}{2};-\frac{w^2}{4}\right)\Bigg|_{z=0}\right)\nonumber\\
&=\frac{1}{w^2}\pi e^{-\frac{w^2}{4}}\mathrm{erfi}^2(w/2)\log(m\alpha)+\frac{\sqrt{\pi}\mathrm{erfi}(w/2)}{2w}\sum_{n=0}^\infty \frac{(-w^2/4)^n}{\left(3/2\right)_n}(\psi(n+1)+\gamma)\nonumber\\
&=\mathcal{C}(iw)\log(m\alpha)+\frac{1}{2}\mathcal{B}(iw),
\end{align}
where in the penultimate step we used \eqref{1F1 -} and in the ultimate step \eqref{cw} and \eqref{bw}. Substitute \eqref{lim of l2} in \eqref{lim in terms of l2}, thereby obtaining
\begin{align}\label{limit of sum}
\lim_{z\to 0}\sum_{m=1}^\infty\varphi_w(z,m\alpha)
&=\sum_{m=1}^\infty\left\{-\psi_w(m\alpha+1)+\frac{\mathcal{C}(w)}{2m\alpha}+\mathcal{C}(iw)\log(m\alpha)+\frac{1}{2}\mathcal{B}(iw)\right\}\nonumber\\
&=-\sum_{m=1}^\infty\lambda_w(m\alpha),
\end{align}
which follows from \eqref{lamdaw}. Equations \eqref{final L with alpha} and \eqref{limit of sum} together prove the claim \eqref{claim}. This completes the proof of the theorem.
\end{proof}

\begin{proof}[Corollary \textup{\ref{w=0 of main thm}}][]
Let $w\to0$ on both sides of \eqref{ramanujanweqn}. Note that $\lim_{w\to0}\mathcal{C}(w)=1$ and $\lim_{w\to0}\mathcal{B}(w)=0$ as $\lim_{w\to0}\frac{\mathrm{erf}(w/2)}{w}=\frac{1}{\sqrt{\pi}}$ and $\psi(1)=-\gamma$. These facts then give $\lim_{w\to0}\lambda_w(x)=\lambda(x)$. This proves \eqref{w1.26}.
\end{proof}

\section{Asymptotics of some integrals with $\Xi(t)$ in their integrands}\label{asymptotics}
Here we prove Theorem \ref{asymp of int} and then recover two results from it, one of which is new.
\begin{proof}[Theorem \textup{\ref{asymp of int}}][]
We start with \eqref{genramhureqeqn},
\begin{align}\label{thm1.5}
&\frac{2^{z-1}\pi^{\frac{z-3}{2}}}{\Gamma(z+1)}\int_0^\infty \Gamma\left(\frac{z-1+it}{4}\right)\Gamma\left(\frac{z-1-it}{4}\right)\Xi\left(\frac{t+iz}{2}\right)\Xi\left(\frac{t-iz}{2}\right)\frac{\Delta_2\left(\alpha,\frac{z}{2},w,\frac{1+it}{2}\right)}{(z+1)^2+t^2}dt\nonumber\\
&=\a^{\frac{z+1}{2}}\left(\sum_{m=1}^{\infty}\varphi_w(z,m\a)-\frac{\zeta(z+1)A_w(z)}{2\alpha^{z+1}}-\frac{\zeta(z)A_w(-z)}{\alpha z}\right).
\end{align}
From \cite[p.~75, Equation (9.2.10)]{dkRiMS1}, we have
\begin{align}\label{alphawala}
&-\frac{1}{\pi}\Gamma\left(\frac{z}{2}\right)\sum_{n=1}^\infty\sigma_{-z}(n) \int_0^\infty e^{\frac{-w^2}{2}}{}_1K_{\frac{z}{2},iw}(2\alpha x)\left({}_2F_1\left(1,\frac{z}{2};\frac{1}{2};-\frac{x^2}{\pi^2n^2}\right)-1\right)x^{\frac{z-2}{2}}\, dx\nonumber\\
&=\frac{\a^{\frac{z}{2}}\G(z+1)e^{-\frac{w^2}{4}}}{2^{z+1}}\sum_{m=1}^{\infty}\varphi_{w}(z,m\a),
\end{align}
where ${}_1K_{z,w}(x)$ is the generalized modified Bessel function introduced in \cite[p.~11, Equation (1.1.37)]{dkRiMS1} by
\begin{align}\label{def}
{}_1K_{z,w}(x)&:=\frac{1}{2\pi i}\int_{(c)}\Gamma\left(\frac{1+s-z}{2}\right)\Gamma\left(\frac{1+s+z}{2}\right) {}_1F_1\left(\frac{1+s-z}{2};\frac{3}{2};-\frac{w^2}{4}\right)\nonumber\\
&\qquad\qquad\times {}_1F_1\left(\frac{1+s+z}{2};\frac{3}{2};-\frac{w^2}{4}\right)2^{s-1}x^{-s}\, ds,
\end{align}
where $z, w \in\mathbb{C}$, $x\in\mathbb{C}\backslash\{x\in\mathbb{R}: x\leq 0\}$ and $c:=$Re$(s)>-1\pm$Re$(z)$. 
Substituting \eqref{alphawala} in \eqref{thm1.5}, we see that
\begin{align}\label{simp}
&\frac{e^{-\frac{w^2}{4}}\pi^{\frac{z-3}{2}}}{2}\int_0^\infty \Gamma\left(\frac{z-1+it}{4}\right)\Gamma\left(\frac{z-1-it}{4}\right)\Xi\left(\frac{t+iz}{2}\right)\Xi\left(\frac{t-iz}{2}\right)\frac{\Delta_2\left(\alpha,\frac{z}{2},w,\frac{1+it}{2}\right)}{(z+1)^2+t^2}dt\nonumber\\
&=-\frac{2\sqrt{\alpha}}{\pi}\Gamma\left(\frac{z}{2}\right)\sum_{n=1}^\infty\sigma_{-z}(n) \int_0^\infty e^{\frac{-w^2}{2}}{}_1K_{\frac{z}{2},iw}(2\alpha x)\left({}_2F_1\left(1,\frac{z}{2};\frac{1}{2};-\frac{x^2}{\pi^2n^2}\right)-1\right)x^{\frac{z-2}{2}}dx\nonumber\\
&\qquad-\frac{\G(z+1)e^{\frac{-w^2}{4}}}{2^{z}}\a^{\frac{(z+1)}{2}}\left(\frac{\zeta(z+1)A_w(z)}{2\alpha^{z+1}}+\frac{\zeta(z)A_w(-z)}{\alpha z}\right).
\end{align}
We first find the asymptotic estimate of the integral inside the sum on the right-hand side of \eqref{simp}. Note that
\begin{align}\label{cov}
&\int_0^\infty {}_1K_{\frac{z}{2},iw}(2\alpha x)\left({}_2F_1\left(1,\frac{z}{2};\frac{1}{2};-\frac{x^2}{\pi^2n^2}\right)-1\right)x^{\frac{z-2}{2}}dx\nonumber\\
&=\frac{1}{(2\alpha)^{z/2}}\int_0^\infty {}_1K_{\frac{z}{2},iw}(x)\left({}_2F_1\left(1,\frac{z}{2};\frac{1}{2};-\frac{x^2}{4\pi^2n^2\alpha^2}\right)-1\right)x^{\frac{z-2}{2}}dx.
\end{align}
Using the series definition of ${}_2F_1$, for $m\in\mathbb{N}$, as $\alpha\to\infty$, we have
$${}_2F_1\left(1,\frac{z}{2};\frac{1}{2};-\frac{x^2}{4\pi^2n^2\alpha^2}\right)-1=\sum_{k=1}^{m-1}\frac{(-1)^k(z/2)_k}{(1/2)_k}\left(\frac{x^2}{4\pi^2n^2\alpha^2}\right)^k+O_{z}\left(\frac{x^{2m}}{\alpha^{2m}n^{2m}}\right).$$
Substituting the above result in \eqref{cov}, we see that as $\alpha\to\infty$
\begin{align}\label{integral bound-1}
&\int_0^\infty {}_1K_{\frac{z}{2},iw}(2\alpha x)\left({}_2F_1\left(1,\frac{z}{2};\frac{1}{2};-\frac{x^2}{\pi^2n^2}\right)-1\right)x^{\frac{z-2}{2}}dx\nonumber\\
&=\frac{1}{(2\alpha)^{z/2}}\left\{\sum_{k=1}^{m-1}\frac{(-1)^k(z/2)_k}{(1/2)_k(4\pi^2n^2\alpha^2)^k}\int_0^\infty x^{\frac{z}{2}+2k-1} {}_1K_{\frac{z}{2},iw}(x)dx\right.\nonumber\\
&\qquad\left.\qquad\quad+O_{z,m}\left(\frac{1}{(n\alpha)^{2m}}\int_0^\infty x^{\frac{z}{2}+2m-1} {}_1K_{\frac{z}{2},iw}(x)dx\right)\right\}\nonumber\\
&=\frac{1}{(2\alpha)^{z/2}}\left\{\sum_{k=1}^{m-1}\frac{(-1)^k(z/2)_k\Gamma\left(k+\frac{1}{2}\right)\Gamma\left(\frac{z+1}{2}+k\right)}{(1/2)_k(4\pi^2n^2\alpha^2)^k}{}_1F_1\left(\frac{z+1}{2}+k;\frac{3}{2};\frac{w^2}{4}\right){}_1F_1\left(\frac{1}{2}+k;\frac{3}{2};\frac{w^2}{4}\right)2^{\frac{z}{2}+2k-1}\right.\nonumber\\
&\qquad\qquad\quad\left.+O_{z,m,w}\left(\frac{1}{(n\alpha)^{2m}}\right)\right\},
\end{align}
where we used the definition of ${}_1K_{\frac{z}{2},iw}(x)$ from \eqref{def} and the fact that integral inside the big -O bound is convergent and independent of $\alpha$. An application of duplication formula of gamma function in \eqref{integral bound-1} yields
\begin{align}\label{integral bound1}
&\int_0^\infty {}_1K_{\frac{z}{2},iw}(2\alpha x)\left({}_2F_1\left(1,\frac{z}{2};\frac{1}{2};-\frac{x^2}{\pi^2n^2}\right)-1\right)x^{\frac{z-2}{2}}dx\nonumber\\
&=\frac{\pi}{\alpha^{z/2}2^{z}}\sum_{k=1}^{m-1}\frac{(-1)^k\Gamma\left(z+2k\right)}{(2\pi n\alpha)^{2k}\Gamma(z/2)}{}_1F_1\left(\frac{z+1}{2}+k;\frac{3}{2};\frac{w^2}{4}\right){}_1F_1\left(\frac{1}{2}+k;\frac{3}{2};\frac{w^2}{4}\right)+O_{z,m,w}\left(\frac{1}{n^{2m}\alpha^{2m+\frac{z}{2}}}\right),
\end{align}
Equations \eqref{simp} and \eqref{integral bound1} along with the fact $\sum_{n=1}^\infty\sigma_a(n)n^{-s}=\zeta(s)\zeta(s-a),\ \mathrm{Re}(s)>\max\{1,\ 1+\mathrm{Re}(a)\}$, imply
\begin{align}
&\frac{e^{-\frac{w^2}{4}}\pi^{\frac{z-3}{2}}}{2}\int_0^\infty \Gamma\left(\frac{z-1+it}{4}\right)\Gamma\left(\frac{z-1-it}{4}\right)\Xi\left(\frac{t+iz}{2}\right)\Xi\left(\frac{t-iz}{2}\right)\frac{\Delta_2\left(\alpha,\frac{z}{2},w,\frac{1+it}{2}\right)}{(z+1)^2+t^2}dt\nonumber\\
&=-\frac{e^{-w^2/2}}{\alpha^{\frac{z-1}{2}}2^{z-1}}\sum_{k=1}^{m-1}\frac{(-1)^k\Gamma\left(z+2k\right)}{(2\pi\alpha)^{2k}}\zeta(2k)\zeta(z+2k){}_1F_1\left(\frac{z+1}{2}+k;\frac{3}{2};\frac{w^2}{4}\right){}_1F_1\left(\frac{1}{2}+k;\frac{3}{2};\frac{w^2}{4}\right)\nonumber\\
&\quad-\frac{\G(z+1)e^{-\frac{w^2}{4}}}{2^{z}}\left(\frac{\zeta(z+1)A_w(z)}{2\alpha^{\frac{z+1}{2}}}+\frac{\zeta(z)A_w(-z)}{ z\alpha^{\frac{1-z}{2}}}\right)+O_{z,m,w}\left(\frac{1}{\alpha^{2m+\frac{z}{2}}}\right).
\end{align}
Finally divide both sides of the above equation by $\frac{1}{2}e^{-\frac{w^2}{4}}\pi^{\frac{z-3}{2}}$ to complete the proof of the theorem.
\end{proof}

\begin{proof}[Corollary \textup{\ref{drz asym}}][]
Note that ${}_1F_1(a;c;0)=1$ and $\Delta_2\left(\alpha,\frac{z}{2},w,\frac{1+it}{2}\right)=2\cos\left(\frac{1}{2}t\log(\alpha)\right)$.
Now \eqref{drz asym eqn} is an easy consequence of \eqref{asymp of int eqn}.
\end{proof}

\begin{proof}[Corollary \textup{\ref{z=0 asym}}][]
Let $z=0$ on both sides of \eqref{asymp of int eqn}. It is easy to get the left-hand side of \eqref{z=0 asym eqn} from the integral in \eqref{asymp of int eqn} up to the factor $e^{\frac{w^2}{4}}$ coming from $\Delta_2\left(\alpha,0,w,\frac{1+it}{2}\right)$. Also, it is easy to take $z=0$ in the finite sum on the right-hand side of \eqref{asymp of int eqn}. Now observe that 
\begin{align}\label{tbr}
\lim_{z\to0}\left\{\frac{\zeta(z+1)A_{w}(z)}{2\alpha^{\frac{1+z}{2}}}+\frac{\zeta(z)A_{w}(-z)}{z\alpha^{\frac{1-z}{2}}}\right\}&=\lim_{z\to0}\left\{\alpha^{\frac{1+z}{2}}\left(\frac{\zeta(z+1)A_{w}(z)}{\alpha^{z+1}}+\frac{\zeta(z)A_{w}(-z)}{z\alpha}\right)\right\}\nonumber\\
&=\sqrt{\alpha}\ L_1\left(w,\alpha\right)\nonumber\\
&=\sqrt{\alpha}\left(\frac{\gamma-\log(2\pi\alpha)}{2\alpha}\mathcal{C}(w)+\frac{1}{2\alpha}\mathcal{B}(w)\right),
\end{align}
where in the second last step we used \eqref{lim-1} and in the last step \eqref{final L with alpha} is invoked. The proof of \eqref{z=0 asym eqn} is then complete upon using the above facts.
\end{proof}

\section*{\textbf{Acknowledgments}}
The first author's research was partially supported by the Swarnajayanti Fellowship Grant SB/SJF/2021-22/08 of SERB (Govt. of India). The second author's research was supported by the Fulbright-Nehru Postdoctoral Fellowship Grant 2846/FNPDR/2022. Both the authors sincerely thank the respective funding agencies for their support.


\end{document}